\documentclass{amsart}

\title{Pseudo-countable models}

\author{Joel David Hamkins}
\address[Joel David Hamkins]
{O'Hara Professor of Philosophy and Mathematics, University of Notre Dame, 100 Malloy Hall, Notre Dame, IN 46556 USA \&\ Associate Faculty Member, Professor of Logic, Faculty of Philosophy, University of Oxford, UK}
\email{jdhamkins@nd.edu}

\urladdr{http://jdh.hamkins.org}

\thanks{Commentary about this article can be made at \href{http://jdh.hamkins.org/pseudo-countable-models}{http://jdh.hamkins.org/pseudo-countable-models}.}

\usepackage{latexsym,amsfonts,amsmath,amssymb,mathrsfs}
\usepackage{bbm}
\usepackage[hidelinks]{hyperref}
\usepackage[cmyk,svgnames,dvipsnames]{xcolor}
\usepackage{tikz}
\usetikzlibrary{arrows,arrows.meta,petri,topaths,positioning,shapes,shapes.misc,patterns,calc,decorations.pathreplacing,hobby}
\pgfdeclarelayer{foreground}
\pgfdeclarelayer{background}
\pgfdeclarelayer{boardarrows}
\pgfdeclarelayer{boardgrid}
\pgfdeclarelayer{boardshades}
\pgfsetlayers{boardshades,boardgrid,boardarrows,background,main,foreground}
\usepackage{wrapfig} 
\usepackage{caption} 
\DeclareCaptionStyle{compact}{font=footnotesize,format=plain,name={Fig},labelsep=period,skip=1.5ex,margin=0pt,oneside,justification=centering}
\DeclareCaptionStyle{leftside}{style=compact,margin={0pt,9pt}}
\DeclareCaptionStyle{rightside}{style=compact,margin={9pt,0pt}}
\usepackage{float}
\usepackage[utf8]{inputenc}
\RequirePackage{doi}

\usepackage{enumitem}
\newtheorem{theorem}{Theorem}
\newtheorem*{theorem*}{Theorem}

\newtheorem*{maintheorem*}{Main Theorem}
\newtheorem*{maintheorems*}{Main Theorems}
\newtheorem{corollary}[theorem]{Corollary}
\newtheorem*{corollary*}{Corollary}
\newtheorem*{corollaries*}{Corollaries}

\theoremstyle{definition}

\newtheorem*{definition*}{Definition}

\newtheorem*{question*}{Question}

\newtheorem*{questions*}{Questions}
\newtheorem*{mainquestion*}{Main Question} 
\newtheorem*{openquestion*}{Open Question} 
\theoremstyle{remark}

\newcommand{\QED}{\end{proof}}

\def\proclaim[#1]{{\bf #1}}
\def\BF#1.{{\bf #1.}}

\def\says#1:#2\par{\item[#1] #2\par}

\newcommand{\Los}{\L o\'s}

\newcommand{\B}{{\mathbb B}}

\renewcommand{\P}{{\mathbb P}}

\newcommand{\overbar}[1]{\mkern 3.5mu\overline{\mkern-3.5mu#1\mkern-.5mu}\mkern.5mu}
\newcommand{\barin}{\mathrel{\mkern3mu\overline{\mkern-3mu\in\mkern-1.5mu}\mkern1.5mu}}

\newcommand{\Mbar}{{\overbar{M}}}
\newcommand{\Nbar}{{\overbar{N}}}
\newcommand{\Vbar}{{\overbar{V}}}

\makeatletter
\newcommand{\dotminus}{\mathbin{\text{\@dotminus}}}
\newcommand{\@dotminus}{%
  \ooalign{\hidewidth\raise1ex\hbox{.}\hidewidth\cr$\m@th-$\cr}%
}
\makeatother

\newcommand{\of}{\subseteq}

\newcommand{\fo}{\supseteq}

\newcommand{\set}[1]{\{\,{#1}\,\}}

\newcommand{\elesub}{\prec}

\newcommand{\Coll}{\mathop{\rm Coll}}

\newcommand{\image}{\mathbin{\hbox{\tt\char'42}}}

\newcommand{\restrict}{\upharpoonright} 
\newcommand{\satisfies}{\models}

\renewcommand{\setminus}{\raise.3ex\hbox{\rotatebox{-20}{$-$}}} 

\renewcommand{\emptyset}{\varnothing}

\newcommand{\intersect}{\cap}

\newcommand{\smalllt}{\mathrel{\mathchoice{\raise2pt\hbox{$\scriptstyle<$}}{\raise1pt\hbox{$\scriptstyle<$}}{\raise0pt\hbox{$\scriptscriptstyle<$}}{\scriptscriptstyle<}}}
\newcommand{\smallleq}{\mathrel{\mathchoice{\raise2pt\hbox{$\scriptstyle\leq$}}{\raise1pt\hbox{$\scriptstyle\leq$}}{\raise1pt\hbox{$\scriptscriptstyle\leq$}}{\scriptscriptstyle\leq}}}

   \def\DHLhksqrt#1#2{%
   \setbox0=\hbox{$#1\sqrt{#2\,}$}\dimen0=\ht0
   \advance\dimen0-0.2\ht0
   \setbox2=\hbox{\vrule height\ht0 depth -\dimen0}%
   {\box0\lower0.4pt\box2}}
\newcommand{\boolval}[1]{\mathopen{\lbrack\!\lbrack}\,#1\,\mathclose{\rbrack\!\rbrack}}
\def\[#1]{\mathopen{\lbrack\!\lbrack}#1\mathclose{\rbrack\!\rbrack}}
\newbox\gnBoxA
\newbox\gnBoxB
\newdimen\gnCornerHgt
\setbox\gnBoxA=\hbox{\tiny$\ulcorner$}
\global\gnCornerHgt=\ht\gnBoxA
\newdimen\gnArgHgt
\def\gcode #1{%
\setbox\gnBoxA=\hbox{$#1$}%
\setbox\gnBoxB=\hbox{$\bar #1$}%
\gnArgHgt=\ht\gnBoxB%
\ifnum     \gnArgHgt<\gnCornerHgt \gnArgHgt=0pt%
\else \advance \gnArgHgt by -\gnCornerHgt%
\fi \raise\gnArgHgt\hbox{\tiny$\ulcorner$} \box\gnBoxA %
\raise\gnArgHgt\hbox{\tiny$\urcorner$}}
\newcommand{\UnderTilde}[1]{{\setbox1=\hbox{$#1$}\baselineskip=0pt\vtop{\hbox{$#1$}\hbox to\wd1{\hfil$\sim$\hfil}}}{}}
\newcommand{\Undertilde}[1]{{\setbox1=\hbox{$#1$}\baselineskip=0pt\vtop{\hbox{$#1$}\hbox to\wd1{\hfil$\scriptstyle\sim$\hfil}}}{}}
\newcommand{\undertilde}[1]{{\setbox1=\hbox{$#1$}\baselineskip=0pt\vtop{\hbox{$#1$}\hbox to\wd1{\hfil$\scriptscriptstyle\sim$\hfil}}}{}}
\newcommand{\UnderdTilde}[1]{{\setbox1=\hbox{$#1$}\baselineskip=0pt\vtop{\hbox{$#1$}\hbox to\wd1{\hfil$\approx$\hfil}}}{}}
\newcommand{\Underdtilde}[1]{{\setbox1=\hbox{$#1$}\baselineskip=0pt\vtop{\hbox{$#1$}\hbox to\wd1{\hfil\scriptsize$\approx$\hfil}}}{}}

\renewcommand{\implies}{\mathrel{\rightarrow}}
\newcommand{\Implies}{\mathrel{\Rightarrow}}

\newcommand{\Iff}{\mathrel{\Longleftrightarrow}}

\newcommand{\iso}{\cong}
\def\<#1>{\left\langle#1\right\rangle}

\newcommand{\val}{\mathop{\rm val}\nolimits}

\newcommand{\ZFC}{{\rm ZFC}}
\newcommand{\ZF}{{\rm ZF}}

\newcommand{\CH}{{\rm CH}}

\newcommand{\HOD}{{\rm HOD}}

\newcommand{\PA}{{\rm PA}}

\newcommand{\cell}[1]{\boxit{\hbox to 17pt{\strut\hfil$#1$\hfil}}}
\newcommand{\head}[2]{\lower2pt\vbox{\hbox{\strut\footnotesize\it\hskip3pt#2}\boxit{\cell#1}}}
\newcommand{\boxit}[1]{\setbox4=\hbox{\kern2pt#1\kern2pt}\hbox{\vrule\vbox{\hrule\kern2pt\box4\kern2pt\hrule}\vrule}}
\newcommand{\Col}[3]{\hbox{\vbox{\baselineskip=0pt\parskip=0pt\cell#1\cell#2\cell#3}}}
\newcommand{\tapenames}{\raise 5pt\vbox to .7in{\hbox to .8in{\it\hfill input: \strut}\vfill\hbox to
.8in{\it\hfill scratch: \strut}\vfill\hbox to .8in{\it\hfill output: \strut}}}
\newcommand{\Head}[4]{\lower2pt\vbox{\hbox to25pt{\strut\footnotesize\it\hfill#4\hfill}\boxit{\Col#1#2#3}}}
\newcommand{\Dots}{\raise 5pt\vbox to .7in{\hbox{\ $\cdots$\strut}\vfill\hbox{\ $\cdots$\strut}\vfill\hbox{\
$\cdots$\strut}}}
\usepackage[backend=bibtex,style=alphabetic,maxbibnames=15,maxcitenames=6,dateabbrev=false]{biblatex}
\bibliography{HamkinsBiblio,MathBiblio,PhilBiblio,WebPosts}
\renewcommand{\UrlFont}{} 
\renewbibmacro{in:}{\ifentrytype{article}{}{\printtext{\bibstring{in}\intitlepunct}}} 
\DeclareFieldFormat{url}{{\UrlFont\url{#1}}} 
\DeclareFieldFormat{urldate}{
  (version \thefield{urlday}\addspace%
  \mkbibmonth{\thefield{urlmonth}}\addspace%
  \thefield{urlyear}\isdot)}
\DeclareFieldFormat{eprint:arxiv}{
\ifhyperref
    {\href{http://arxiv.org/abs/#1}{%
    arXiv\addcolon\nolinkurl{#1}}\iffieldundef{eprintclass}{}{\UrlFont{\mkbibbrackets{\thefield{eprintclass}}}}}
    {arXiv\addcolon\nolinkurl{#1}\iffieldundef{eprintclass}{}{\UrlFont{\mkbibbrackets{\thefield{eprintclass}}}}}}
\newcommand\ZFbar{\overline\ZF}
\newcommand\ZFCbar{\overline\ZFC}

\begin{document}

\begin{abstract}
Every mathematical structure has an elementary extension to a pseudo-countable structure, one that is seen as countable inside a suitable class model of set theory, even though it may actually be uncountable. This observation, proved easily with the Boolean ultrapower theorem, enables a sweeping generalization of results concerning countable models to a rich realm of uncountable models. The Barwise extension theorem, for example, holds amongst the pseudo-countable models---every pseudo-countable model of \ZF\ admits an end extension to a model of $\ZFC+V=L$. Indeed, the class of pseudo-countable models is a rich multiverse of set-theoretic worlds, containing elementary extensions of any given model of set theory and closed under forcing extensions and interpreted models, while simultaneously fulfilling the Barwise extension theorem, the Keisler-Morley theorem, the resurrection theorem, and the universal finite sequence theorem, among others.
\end{abstract}

\maketitle

\section{Introduction}

Every mathematical structure $M$ has an elementary extension \mbox{$M\elesub\Mbar$} to what I shall call a \emph{pseudo-countable} structure $\Mbar$, one that is viewed as countable inside a certain class model of set theory $\mathcal{V}$. This can be seen by means of the Boolean ultrapower theorem, which constructs such models $\mathcal{V}$ as quotients of the Boolean-valued structures $V^\B/U$ arising with the forcing $\B$ collapsing $M$ to become countable; the elementary embedding of $M$ into $\Mbar$ is simply the restriction of the Boolean ultrapower embedding itself. (The entire Boolean-ultrapower construction takes place in $V$, with no need to form any actual forcing extension of the universe.)
Because the uncountable structure $\Mbar$ is seen as countable in $\mathcal{V}$, it falls under the scope there of any theorem on countable models, thereby providing immediate generalizations of results about countable models to an enormous realm of uncountable models. The class of pseudo-countable models of set theory is revealed in this way as a robust multiverse of set-theoretic worlds, one containing elementary extensions of every model of set theory and closed under forcing, class forcing, and interpreted models, while simultaneously fulfilling many attractive features holding in the realm of countable models. In this article I shall explain how this method plays out with several results in the model theory of arithmetic and set theory, such as the Barwise extension theorem and its generalizations to the resurrection theorem and the universal finite sequence theorem.\goodbreak

\section{Review of the Boolean ultrapower}\label{Section.Review-Boolean-ultrapower}

Let me review the basics of the Boolean ultrapower, described at greater length in \cite{HamkinsSeabold:BooleanUltrapowers}. The Boolean ultrapower is intimately connected with the Boolean-valued model approach to forcing, known since the 1960s in work of \cite{Vopenka1965:OnNablaModelOfSetTheory}, Scott, and Solovay, as explicated later in \cite{Bell1985:BooleanValuedModelsAndIndependenceProofs} and now in most standard set theory texts.\goodbreak

\begin{theorem}[Boolean ultrapower theorem]\label{Theorem.Boolean-ultrapower}
 For any forcing notion $\B$, there is an elementary embedding of the set-theoretic universe $V$ to a definable class model $\<\Vbar,\barin>$, not necessarily well founded, such that (in $V$) there is a $\Vbar$-generic filter $G$ for the forcing $\bar\B=j(\B)$.
 $$j:V\to\Vbar\of\Vbar[G]$$
 The embedding $j$ and the models $\<\Vbar,\barin>$ and $\<\Vbar[G],\barin>$ are all definable classes in~$V$.
\end{theorem}

\begin{proof}
This is exactly what the Boolean ultrapower provides. Consider any forcing notion $\B$, a complete Boolean algebra. Let $V^\B$ be the usual class of all $\B$ names. This is well known to be a $\B$-valued model of \ZFC\ set theory, with the Boolean values $\boolval{\varphi}$ defined as a member of $\B$ for any assertion $\varphi$ in the forcing language, with includes the membership relation $\in$ and constant symbols for every $\B$-name, as well as a predicate $\check V$ for the ground model, defined by
  $$\boolval{\tau\in\check V}\quad =\quad \bigvee_{x\in V}\boolval{\tau=\check x}.$$
Every axiom of \ZFC\ has Boolean value $1$. 

Let $U\of\B$ be any ultrafilter on $\B$ in $V$, and let me remark specifically that there is no need in this construction for $U$ to be generic in any way. Indeed, one should specifically use $U\in V$ in order to ensure that the models and the embedding are definable classes in $V$, definable from parameters $\B$ and $U$. We define the corresponding equivalence and membership relations modulo $U$ by
\begin{align*}
  \sigma=_U\tau\quad &\Iff\quad\boolval{\sigma=\tau}\in U \\
  \sigma\in_U\tau\quad &\Iff\quad\boolval{\sigma\in\tau}\in U.
\end{align*}
The equivalence relation $=_U$ is a congruence with respect to $\in_U$, and one may form the Boolean quotient $V^\B/U$ to consist of the equivalence classes
 $$[\sigma]_U=\set{\tau\in V^\B\mid\sigma=_U\tau},$$
with the membership relation induced by $\in_U$
 $$[\sigma]_U\barin[\tau]_U\quad\Iff\quad \sigma\in_U\tau\quad\Iff\quad\boolval{\sigma\in\tau}\in U.$$
Kindly note that the quotient structure $V^\B/U$ is not generally the same as or isomorphic to the result of the recursively defined value assignments $\val(\tau,U)=\set{\val(\sigma,U)\mid \sigma\in_U \tau}$, a construction that works properly only when $U$ is $V$-generic.

With the quotient construction as above, the relevant \Los\ theorem establishes
 $$V^\B/U\satisfies\varphi\quad\text{ if and only if }\quad\boolval{\varphi}\in U.$$
Thus, any statement forced by $\B$ is true in the structure $\<V^\B/U,\barin>$, which is a structure definable in $V$ from parameters $\B$ and $U$. In this way, the forcing construction can be seen to provide an interpretation of the theory that is forced. The method therefore reveals a robust mutual interpretability phenomenon in set theory---the theories $\ZFC$ and $\ZFC+\neg\CH$, for example, are not merely equiconsistent but mutually interpretable, and there is no need to undertake forcing with countable transitive models or to undertake metatheoretic argument with the reflection theorem, as one commonly sees in older treatments of forcing.

The Boolean ultrapower embedding now arises from this situation by considering more carefully the ground model predicate $\check V$ and its quotient by the ultrafilter $U$. Namely, let $\Vbar=\check V/U$, by which I mean the class of equivalence classes $[\tau]_U$ of names $\tau\in V^\B$ for which $\boolval{\tau\in\check V}\in U$. This is not the same as the class of check names $[\check x]_U$, because some names $\tau$ can be mixtures of such check names along a maximal antichain that is not met by $U$, which exists since $U$ is not generic, and these individuals $[\tau]_U$ will be in $\Vbar$ but not of the form $[\check x]$ for any $x$.\goodbreak

Nevertheless, the map
 $$j:V\to\Vbar\qquad\qquad j:x\mapsto[\check x]_U$$
is an elementary embedding $j:V\to \Vbar$ known as the \emph{Boolean ultrapower} map, explicated at length in \cite{HamkinsSeabold:BooleanUltrapowers}.\footnote{Unfortunately, the Boolean ultrapower map is less well known than it deserves. It appears to be missing, for example, in all the standard accounts of the Boolean-valued-model approach to forcing, which give a full account of $V^\B/U$, but do not mention the Boolean ultrapower map. Furthermore, experienced researchers deeply familiar with $V^\B$ and $V^\B/U$ sometimes appear unaware of the underlying elementary embedding $j_U:V\to\check V/U$, a situation at times exacerbated by regrettable confusion concerning whether $U$ must be generic in the construction of $V^\B/U$.}

Finally, the equivalence class of the canonical name $\dot G$ for the generic filter is represented by the object $G=[\dot G]_U$ in $V^\B/U$, with a full Boolean value asserting that it is generic, since $\boolval{\dot G\intersect \check D\neq\emptyset}=1$ for every dense set $D\of\B$ and consequently $\boolval{\dot G\text{ is }\check V\text{-generic}}=1$. Furthermore, $\boolval{\sigma=\val(\check\sigma,\dot G)}=1$ for every $\B$-name $\sigma$, which shows that $V^\B$ thinks every object is the interpretation of a name by $\dot G$ and that the universe is a forcing extension of the ground model $\check V$ by $\dot G$. It follows by the \Los\ theorem that $G$, defined as the equivalence class $[\dot G]_U$ (which exists in $V$) is $\Vbar$-generic for the forcing $j(\B)$, as desired.
\end{proof}

For emphasis, let me state again categorically that although we have mentioned a forcing notion $\B$, nevertheless the models $\<\Vbar,\barin>$ and $\<\Vbar[G],\barin>$ and the map $j:V\to\Vbar$ are definable classes in $V$, definable with $\B$ and $U$ as parameters. There is no need in this construction to form any actual forcing extension of the set-theoretic universe $V$.

The following special instance of the Boolean ultrapower theorem will play a central role in the later applications.

\begin{corollary}\label{Corollary.Boolean-ultrapower-kappa}
 For any cardinal $\kappa$, there is a definable elementary embedding of the set-theoretic universe $V$ to a definable class model $\<\Vbar,\barin>$ for which there is (in $V$) a $\Vbar$-generic filter $G$ for the forcing to collapse $j(\kappa)$ to become countable.
 $$j:V\to\Vbar\of\Vbar[G]$$
The cardinal $j(\kappa)$ is thus countable in $\Vbar[G]$, and every structure $M$ in $V$ of size $\kappa$ has $j(M)$ countable in $\Vbar[G]$.
\end{corollary}

\begin{proof}
Let $\B=\Coll(\omega,\kappa)$ be the forcing notion to collapse $\kappa$ to $\omega$, and let $U\of\B$ be any ultrafilter on this Boolean algebra in $V$. By the Boolean ultrapower theorem, the corresponding Boolean ultrapower map $j:V\to\Vbar\of\Vbar[G]$ is an elementary embedding of $V$ into the model $\Vbar=\check V/U$, and $G=[\dot G]_U$ is a $\Vbar$-generic filter on $j(\B)$, which collapses $j(\kappa)$ to become countable.
\end{proof}

\section{Pseudo-countable models}

I define that a mathematical structure $M$ is \emph{pseudo countable}, if it is isomorphic to a structure seen as countable inside some Boolean-quotient model $V^\B/U$, where $\B=\Coll(\omega,\kappa)$ is the forcing to collapse some cardinal $\kappa$ to become countable and $U$ is an ultrafilter on $\B$ in $V$. This includes every countable structure, since in the case $\kappa=1$ the forcing $\B$ is trivial and so $V^\B/U$ is isomorphic to $V$ itself. In light of corollary \ref{Corollary.Boolean-ultrapower-kappa} and theorem \ref{Theorem.Every-model-extends-to-psuedo-countable-model} below, however, there are pseudo-countable structures of arbitrarily large uncountable size.

\begin{theorem}\label{Theorem.Every-model-extends-to-psuedo-countable-model}
 Every structure $M$ in a finite language, including uncountable structures of enormous cardinality, has an elementary extension to a pseudo countable structure $\Mbar$.
\end{theorem}

\begin{proof}
 For any structure $M$ of any cardinality, let $\B$ be the forcing to collapse $M$ to become countable, and let $U\of\B$ be any ultrafilter. Consider the associated Boolean ultrapower embedding
  $$j:V\to\Vbar\of\Vbar[G]=V^\B/U$$
 as in corollary \ref{Corollary.Boolean-ultrapower-kappa}. Since $j$ is elementary from $V$ to $\Vbar$, it follows that $j\restrict M:M\to j(M)$ is an elementary embedding, since if $M\satisfies\varphi[a]$ then this is visible in $V$ and hence $j(M)\satisfies\varphi[j(a)]$ inside $\Vbar$. And since $j(M)$ is countable in $\Vbar[G]$, the structure $j(M)$ is pseudo countable. By identifying $M$ with its image, we may thus find an elementary extension of $M\elesub\Mbar\iso j(M)$ to a pseudo countable structure, as desired.
\end{proof}

There is a certain subtlety regarding the language of the structures $M$ and $\Mbar$. Namely, if $M$ is an $\mathcal L$-structure, then of course $\Mbar=j(M)$ becomes a $j({\mathcal L})$-structure in $\Vbar$, since the linguistic resources of the language signature may be transformed by $j$. For this reason, it would have been more correct for me to write $j(M)\satisfies j(\varphi)[j(a)]$ in the proof of theorem \ref{Theorem.Every-model-extends-to-psuedo-countable-model}, since the formula $\varphi$ is transformed by $j$ to a $j({\mathcal L})$ assertion $j(\varphi)$. Nevertheless, allowing this kind of transformation of the language, we may regard $\Mbar$ as an $\mathcal L$ structure, replacing each component of the signature of $\mathcal L$ with its associated version in $j({\mathcal L})$. This way of thinking is especially successful when $\mathcal L$ is a finite language, as I had stated in the theorem, for in this case all the relation, function, and constant symbols of $j({\mathcal L})$ arise as translations of a corresponding element of $\mathcal{L}$. If the language $\mathcal{L}$ is infinite, however, even just countably infinite, then with nontrivial forcing the critical point of $j$ will necessarily be $\omega$, and so there will be new natural numbers in $\Vbar$ not in the range of $j$. In this case there will consequently also be new components of $j(\mathcal{L})$ not arising directly as a correspondent of some language component of $\mathcal{L}$. That is, the structure $\Mbar$ as a $j(\mathcal{L})$-structure interprets parts of the signature not directly analogous to any particular component of $\mathcal{L}$. Nevertheless, in such cases we can still view $\Mbar$ as an $\mathcal{L}$-structure simply by restricting to the language arising from $j\image\mathcal{L}$. The subtlety is that this reduct structure, however, does not exist in $\Vbar$. Nevertheless, there is no need actually to perform the reduction like that inside $\Vbar$, if one wants only to see that $M\satisfies\varphi[a]\Implies\Mbar\satisfies\varphi[j(a)]$, since the formula $\varphi$ uses only finitely many of the resources, and so it suffices to take the reduct to the finite language appropriate for $j(\varphi)$, which involves only the parts of the language $\mathcal{L}$ arising from $\mathcal{L}$. So let us keep these subtleties about the language in the back of our minds as we proceed.

Let me mention further that this feature of $\Vbar$ having new natural numbers is of course a necessary feature of the main conclusion of theorem \ref{Theorem.Every-model-extends-to-psuedo-countable-model} in the case that $M$ is uncountable, since in order for $j(M)$ to become countable in $\Vbar[G]$, which is a class in $V$, it must be that the predecessors of $\omega$ in $\Vbar$ are equinumerous in $V$ with the predecessors of $j(M)$, of which there are at least $|M|$ many. So in fact, there will be at least $|M|$ many natural numbers in $\Vbar$, even though $\Vbar$ views them as the (countable) set of natural numbers $\omega^\Vbar$. In short, the reason $j(M)$ is able to become countable in $\Vbar[G]$ is because $\omega^{\Vbar}$ was pumped up to the same size.

\section{The multiverse of pseudo-countable models}

Let $\mathcal{M}$ be the class of all such pseudo countable models of \ZF\ set theory, a class I shall refer to as the \emph{multiverse of pseudo countable models of set theory}. 

Notice the subtle point that a model $M$ of $\ZF$ is pseudo-countable when it is seen as countable inside some $V^\B/U$, but this doesn't mean that $V^\B/U$ needs to agree that it is a model of \ZF, since perhaps it only satisfies some nonstandard fragment of $\ZF$ in $V^\B/U$. When the forcing is nontrivial, the model $V^\B/U$ is necessarily $\omega$-nonstandard, after all, and so it has a nonstandard understanding of the theory \ZF. For example, $V^\B/U$ will think that a given model $M$ of \ZF\ will satisfy the reflection theorem using nonstandard finite fragments of its version of \ZF, so that $(V_\alpha)^M$ will not be models of \ZF\ in the sense of $V^\B/U$, even though they actually satisfy every axiom of \ZF.

\begin{theorem}\label{Theorem.Multiverse-of-pseudo-countable-models}
Let $\mathcal{M}$ be the multiverse of pseudo countable models of \ZF\ set theory.
 \begin{enumerate}
   \item $\mathcal{M}$ is closed under forcing, in that if $M\in\mathcal{M}$ and $\P$ is a notion of forcing in $M$, then there is some $M$-generic filter $G\of\P$ for which the forcing extension $M[G]$ is also in $\mathcal{M}$.
   \item $\mathcal{M}$ is also closed under class forcing, using any class forcing notion $\P$ definable in such an $M$. 
   \item $\mathcal{M}$ is closed under inner models, in that if $M\in\mathcal{M}$ and $W$ is a definable inner model of $M$ and a model of \ZF, then $W\in\mathcal{M}$.
   \item More generally, $\mathcal{M}$ is closed under interpreted models, in that if $M\in\mathcal{M}$ and $W$ is a model of \ZF\ set theory that is interpreted in $M$, then $W\in\mathcal{M}$.
 \end{enumerate}
\end{theorem}

\begin{proof}
These closure properties are all very clear upon reflection, since if $M$ is pseudo countable, then it is a countable model of set theory inside some Boolean quotient $V^\B/U$, in which case we can also construct all the various forcing extensions and interpreted models, just as we would with countable models, by applying the standard constructions within that world. So the multiverse of pseudo countable models will exhibit all these closure properties.
\end{proof}

In subsequent joint work, Victoria Gitman and I are investigating the multiverse of nontrivially pseudo-countable models as a possible model of the multiverse axioms in the style of \cite{GitmanHamkins2010:NaturalModelOfMultiverseAxioms,Hamkins2012:TheSet-TheoreticalMultiverse}. In this article, however, I take theorems \ref{Theorem.Every-model-extends-to-psuedo-countable-model} and \ref{Theorem.Multiverse-of-pseudo-countable-models} to establish, assuming the consistency of \ZFC, that the multiverse of psuedo-countable models of set theory is extremely rich, while containing many models of arbitrarily large cardinality.\goodbreak

\section{Generalizing results from countable models to the uncountable}

Let me now explain how the Boolean ultrapower construction and the multiverse of psuedo-countable models enables a uniform generalization of results about countable models to the uncountable. The specific idea is that, precisely because the pseudo-countable models are seen as countable inside the various Boolean quotient models $V^\B/U$, they will be subject in those worlds to the countable-model theorems such as the Barwise extension theorem, the resurrection theorem, and the universal finite sequence theorems. In this way, we will achieve versions of those theorems for the pseudo-countable models, which by theorem \ref{Theorem.Every-model-extends-to-psuedo-countable-model} includes models of arbitrarily large uncountable size, extending any given model.

Recall that a model of set theory $N$ is an \emph{end-extension} of another model $M$, if $M$ is a submodel of $N$ and no set in $M$ gains new members in $N$, so that $a\in^N b\in M\implies a\in M$ and $a\in^M b$. The model $N$ is a \emph{top-extension} of $M$ (also known as a \emph{rank-extension}), in contrast, if furthermore all new individuals of $N$ have higher rank than any ordinal in $M$; in other words, the cumulative hierarchy as computed in $N$ agrees on those ordinals with the cumulative hierarchy of $M$, in that $(V_\alpha)^N=(V_\alpha)^M$ for every ordinal $\alpha$ in $N$.

\subsection{The Barwise extension theorem for pseudo-countable models}

The Barwise extension theorem \cite{Barwise1971:InfinitaryMethodsInTheModelTheoryOfSetTheoryLC69} asserts that every countable model of set theory $M\satisfies\ZF$ has an end-extension $M\of N$ to a model $\<N,\in^N>$ of $\ZFC+V=L$. The theorem does not hold for uncountable models, since if $M$ is well-founded beyond true $\omega_1$, then no end extension will collapse this ordinal and consequently there will be no possibility of changing which reals are thought to be constructible; thus, if there are nonconstructible reals in $M$, then will remain nonconstructible in any end-extension. Nevertheless, the Barwise extension theorem does hold amongst the pseudo-countable models.

\begin{theorem}\label{Theorem.Barwise-uncountable}
 The Barwise extension theorem holds for all pseudo-countable models.
 \begin{enumerate}
   \item Every pseudo-countable model $M\satisfies\ZF$ has an end extension to a pseudo-countable model of $\ZFC+V=L$.
   \item Consequently, every model $M\satisfies\ZF$ has an elementary extension to a model with an end extension satisfying $\ZFC+V=L$.
      $$\forall M\satisfies\ZF\ \exists\Mbar,N\ (M\elesub\Mbar\of_e N\satisfies\ZFC+V=L)$$
   \item Equivalently, every model $M\satisfies\ZF$ admits a $\Delta_0$-elementary embedding into a model of $\ZFC+V=L$.
 \end{enumerate}
\end{theorem}

\begin{proof}
For statement (1), consider any pseudo-countable model $M\satisfies\ZF$. So $M$ is seen as countable inside some Boolean quotient model $V^\B/U$, where $\B$ is collapse forcing of some cardinal to $\omega$. Since the Barwise extension theorem holds inside $V^\B/U$, there is a countable end extension $M\of_e N\satisfies\ZFC+V=L$ inside $V^\B/U$. Since $V^\B/U$ thinks $N$ is an end-extension of $M$, it really is, and so we have fulfilled the Barwise extension theorem inside the class of pseudo-countable models.

For statement (2), we simply apply theorem \ref{Theorem.Every-model-extends-to-psuedo-countable-model} before applying the previous argument. The model $M$ has an elementary extension $M\elesub\Mbar$ to a pseudo-countable model $\Mbar$, which has the desired end extensions $N$.

To prove statement (3), simply compose the elementary embedding $M\elesub\Mbar$ with the end-extension $\Mbar\of_e N$ to have such a $\Delta_0$-elementary embedding. Conversely, if $M$ admits a $\Delta_0$-elementary extension $M\elesub_{\Delta_0} N\satisfies\ZFC+V=L$, then let $\Mbar$ be the downward closure of $M$ in $N$, which by Gaifman's theorem \cite{Gaifman1974:ElementaryEmbeddingsOfModelsOfSetTheoryAndCertainSubtheories} is a cofinal elementary extension of $M$ and end-extended to $N$. So the properties stated in (2) and (3) are equivalent.\end{proof}\goodbreak

Thus we have generalized the Barwise extension theorem far beyond the countable models, to the multiverse of pseudo-countable models, a multiverse of models of arbitrarily large uncountable size, including elementary extensions of any given model of set theory. 

I should like to remark that the particular consequences stated in (2) and (3) are also easily proved without the Boolean ultrapower method by means of a simple model-theoretic compactness argument. Namely, if we consider any uncountable model $M\satisfies\ZF$, then let $T$ be the theory consisting of the $\Delta_0$-elementary diagram of $M$ together with $\ZFC+V=L$. This theory is finitely consistent, since any finite subtheory mentions only finitely many constants from $M$, and so that part of the diagram is true in some countable elementary substructure of $M$, which by the Barwise extension theorem has an end-extension to a model of $\ZFC+V=L$. So any given finite piece of the theory (and even any countable piece) is realized in such an extension. So the theory has a model $N$, which satisfies $\ZFC+V=L$, and $M$ has a $\Delta_0$-elementary embedding into $N$. As argued earlier with Gaifman's theorem, the model $\Mbar$ consisting of the downward closure in $N$ of the constants naming elements of $M$ will be an elementary extension of $M$, and $N$ is an end-extension of $\Mbar$, realizing the desired extension.

The real content of theorem \ref{Theorem.Barwise-uncountable}, therefore, is statement (1), establishing the Barwise end extension property for the class of pseudo-countable models. One can iterate the result within that class. For example, one can extend any given model $M\satisfies\ZF$ to a pseudo-countable model $\Mbar$, afterward end extending to a model $N\satisfies V=L$, and then perform forcing $N[G]$ and then again end-extending to $N\of_e\Nbar\satisfies V=L$, doing so iteratively while staying within the class of pseudo-countable models.

\subsection{The Kiesler-Morley theorem for pseudo-countable models}

The Keisler-Morley \cite{KeislerMorley1968:ElementaryExtensionsOfModelsOfSetTheory} theorem asserts that every countable model $M\satisfies\ZF$ admits a nontrivial elementary end extension $M\elesub M^+$. This theorem does not hold for all uncountable models. For example, if $\kappa$ is the least inaccessible cardinal, then $V_\kappa$ has no elementary end extension. Meanwhile, the theorem does hold for all pseudo-countable models.

\begin{theorem}
  The Kiesler-Morley theorem holds for all pseudo-countable models of set theory. Every pseudo-countable $M\satisfies\ZF$ admits an elementary end extension $M\elesub M^+$.
\end{theorem}

\begin{proof}
  If $M\satisfies\ZF$ is pseudo-countable, then $M$ is seen as a countable model of (a nonstandard fragment of) set theory inside some $V^\B/U$. Applying the Keisler-Morley theorem inside $V^\B/U$---and note that it does apply to sufficient fragments of \ZF---we find the desired elementary end extension $M\elesub M^+$ inside $V^\B/U$, as desired. 
\end{proof}

\subsection{The resurrection theorem for pseudo-countable models}

Kameryn Williams and I proved a certain generalization of the Barwise extension theorem in \cite{HamkinsWilliams2021:The-universal-finite-sequence}, namely, the inner-model resurrection theorem, asserting that for every countable model $M$ of \ZF\ set theory, any c.e. theory $\ZFbar$ extending $\ZF$ that holds in some inner model of $M$ is realized again in an end extension of $M$. (The result is highlighted in corollaries 9 and 10 of \cite{HamkinsWilliams2021:The-universal-finite-sequence}, in a stronger form than I have just stated here, and it is also an immediate consequence of the main theorem of that article.) In short, theories true in inner model of $M$, even if destroyed in $M$, can nevertheless be resurrected in an end extension of the model. The Barwise extension theorem, of course, is simply an instance of this with the theory $\ZFC+V=L$. 

The Boolean ultrapower method immediately extends the recurrence theorem from countable models to the pseudo-countable models. 

\begin{theorem}\label{Theorem.Resurrection-uncountable}
 The recurrence theorem holds for pseudo-countable models. Namely, if $M\satisfies\ZF$ is a pseudo-countable model of set theory, then any c.e. theory $\ZFbar$ extending \ZF\ that is true in some inner model $W$ of $M$ is also true again in a pseudo-countable end-extension $M\of_e N\satisfies\ZFbar$.
\end{theorem}

\begin{proof}
The recurrence theorem holds for the countable models inside any particular $V^\B/U$, with the result that any model of \ZF\ seen as countable in such a world will have the desired end-extension. (One needs to verify that the recurrence theorem holds for countable models of sufficient fragments of $\ZF$, since as I have mentioned, a pseudo-countable model of \ZF\ is countable inside some $V^\B/U$ and may satisfy only a nonstandard fragment of \ZF\ in that model, but this can be handled with the notion of \emph{suitable} theory in \cite{HamkinsWilliams2021:The-universal-finite-sequence}.) Every psuedo-countable model will thus be subject to the recurrence theorem as desired.
\end{proof}

Again it follows as in theorem \ref{Theorem.Barwise-uncountable} statement (2) that every model $M\satisfies\ZF$ with an inner model $W$ of a suitable theory $\ZFbar$ will have an elementary extension $M\elesub\Mbar$, which itself has an end-extension $\Mbar\of_e N$ to a model of the desired theory $N\satisfies\ZFbar$, or as in statement (3) that $M$ has a $\Delta_0$-elementary embedding into such a model $N$. This one-off consequence instance, however, is again easily established without the Boolean ultrapower method by means of a model-theoretic compactness argument. Namely, for any $M\satisfies\ZF$ with an inner model $W\satisfies\ZFbar$, let $T$ be the $T$ be the theory $\ZFbar$ together with the $\Delta_0$ diagram of $M$. This theory is finitely consistent, since any finite part of the diagram is realized in a countable submodel of $M$, to which the resurrection theorem applies, and so there is an end-extension realizing $\ZFbar$ and preserving that part of the diagram. So $T$ is consistent, and any model $N\satisfies T$ admits a $\Delta_0$-elementary embedding from $M$. By Gaifman's theorem, we can decompose this embedding into an elementary embedding $M\elesub\Mbar$ and an end-extension $\Mbar\of N$, as desired for the theorem.

\subsection{The universal finite sequence theorems for pseudo-countable models}

Let me now similarly generalize the various results on the universal finite sequences, extending the universal algorithm; see \cite{Woodin2011:A-potential-subtlety-concerning-the-distinction-between-determinism-and-nondeterminism, Hamkins:The-modal-logic-of-arithmetic-potentialism, BlanckEnayat2017:Marginalia-on-a-theorem-of-Woodin, Blanck2017:Dissertation:Contributions-to-the-metamathematics-of-arithmetic, HamkinsWoodin:The-universal-finite-set, HamkinsWilliams2021:The-universal-finite-sequence, Hamkins:Every-countable-model-of-arithmetic-or-set-theory-has-a-pointwise-definable-end-extension}. 

The first set-theoretic instance of the universal finite sequence phenomenon was the main result of \cite{HamkinsWoodin:The-universal-finite-set}, which shows that there is a $\Sigma_2$-definable \ZF-provably finite sequence $$a_0,a_1,\ldots,a_n$$ with the \emph{universal extension property} for top-extensions, meaning that in any countable model $M\satisfies\ZF$ in which the sequence is $s$, then for any finite extension $t\fo s$ of the sequence in $M$, there is a top-extension $M\of N$ such that the sequence defined in $N$ is $t$.\goodbreak

The main result of \cite{HamkinsWilliams2021:The-universal-finite-sequence} subsequently provided the $\Sigma_1$ analogue of that result, showing that there is a $\Sigma_1$-definable \ZF-provably finite sequence
 $$a_0,a_1,\ldots,a_n$$ with the universal extension property for end-extensions, meaning that in any countable model $M\satisfies\ZF$ in which the sequence is $s$, then for any finite extension $t\fo s$, there is an end-extension $M\of N$ such that the sequence defined in $N$ is $t$. In \cite{HamkinsWilliams2021:The-universal-finite-sequence} we also simultaneously achieved the resurrection property in the extension model $N$, so that $N\satisfies\ZFbar$ for any c.e. theory $\ZFbar$ extending \ZF\ and holding in some inner model of $M$.
 
Meanwhile, the result of \cite[theorem~8]{Hamkins:Every-countable-model-of-arithmetic-or-set-theory-has-a-pointwise-definable-end-extension} extended those results by providing a higher-complexity definition with the universal extension property for $\Sigma_m$-elementary end-extensions (note that these are top-extensions if $m\geq 1$). Namely, for any particular natural number $m$ and any c.e. theory $\ZFCbar$ extending $\ZFC$, there is a $\Sigma_{m+1}$-definable \ZF-provably finite sequence of ordinals
 $$\alpha_0,\alpha_1,\ldots,\alpha_n$$
with the universal extension property for $\Sigma_m$-elementary extensions, meaning that for countable model $M\satisfies\ZFCbar$ in which the sequence is $s$, then for every finite $t$ extending $s$ in $M$, there is a $\Sigma_m$-elementary end extension $M\elesub_{\Sigma_m} N\satisfies\ZFCbar$ in which the sequence defined is $t$. If one has $V=\HOD$, then the ordinal sequence can be used to define a universal finite sequence of arbitrary sets.

All these previous theorems establish the universal finite sequence result for countable models of set theory. Here, I extend the phenomenon to the class of pseudo-countable structures.

\begin{theorem}\label{Theorem.Sigma_m-universal-class-of-models}
The universal finite sequence theorems hold for all pseudo-countable models. Namely, for any particular natural number $m$ and any c.e. theory $\ZFCbar$ extending \ZFC, if $M$ is any pseudo-countable model of $\ZFCbar$ and the $\Sigma_{m+1}$-definable universal finite sequence in $\Mbar$ is $s$, then for any finite extension $t\fo s$ in $M$ there is a $\Sigma_m$-elementary end extension $\Mbar\elesub_{\Sigma_m}N$ to a pseudo-countable model $N$ in which the sequence defined is $t$.
\end{theorem}

\begin{proof}
We simply apply the theorems mentioned earlier inside the set-theoretic world $V^\B/U$, where the pseudo-countable model $M$ is seen as actually countable. The desired end-extension $N$ can therefore also be found in $V^\B/U$.
\end{proof}

Let us draw out the consequence for individual models. 

\begin{corollary}
 For any c.e. theory $\ZFCbar$ extending \ZFC, every model $M\satisfies\ZFCbar$ has an elementary extension $M\elesub\Mbar$ such that for any particular natural number $m$, if the $\Sigma_{m+1}$-definable universal finite sequence in $M$ and $\Mbar$ is $s$, then for any finite extension $t\fo s$ in $\Mbar$, there is an end extension $\Mbar\of_e N\satisfies\ZFCbar$ such that the universal finite sequence defined in $N$ is $t$.
\end{corollary}

\begin{proof}
 The point is that if $j:V\to\Vbar\of\Vbar[G]$ is the Boolean ultrapower arising from an ultrafilter $U$ on the forcing $\B$ to collapse $M$ to become countable, then $\Mbar=j(M)$ will be a pseudo-countable model with all the desired properties. It is an elementary extension of $M$ via $j$, and it is countable in $\Vbar[G]$, where it will therefore fall under the universal finite sequence extension theorem. So all the desired extension models $N$ of $\Mbar$ can also be found in $\Vbar[G]$ for any desired $t$.
\end{proof}

Using the model-theoretic compactness arguments we had earlier, one would be able to achieve an extension $N$ realizing any fixed $t$, but this corollary produces a single elementary extension $\Mbar$ such that all the various sequences $t$ can be realized in end-extensions of $\Mbar=j(M)$. In other words, you only need to make one elementary extension $M\elesub\Mbar$, after which all the further extensions will be end-extensions. So the corollary seems a bit stronger than what is possible to achieve easily with the compactness-style argument. In any case, theorem \ref{Theorem.Sigma_m-universal-class-of-models} establishes the universal finite sequence extension property throughout the multiverse of pseudo-countable models of set theory, a robust realm including many uncountable models.

\section{Pseudo-pointwise-definability}

Results in \cite{Hamkins:Every-countable-model-of-arithmetic-or-set-theory-has-a-pointwise-definable-end-extension} show that every countable model of arithmetic or set theory has a pointwise definable end extension, one in which every individual is definable without parameters. Applying these theorems in the $\omega$-nonstandard realm $V^\B/U$, we see that every pseudo-countable model has an end extension to a pseudo-countable model that is thought in $V^\B/U$ to be pointwise definable. In other words, every model $M$ of \PA\ or \ZF\ admits an elementary extension $M\elesub \Mbar$ to a pseudo-countable model, which has an end extension $M\of_e N$ to a model $N$ that is \emph{pseudo-pointwise-definable}, a model viewed as pointwise definable inside $V^\B/U$ for some collapse forcing $\B$ and ultrafilter $U\of\B$. 

The model $N$, of course, may have vast uncountable size in $V$, and it is therefore not actually pointwise definable. Rather, what is going on is that the $\omega$-nonstandard universe $V^\B/U$ has what it thinks is the satisfaction relation for $N$, and it seems in that nonstandard universe that every individual in $N$ satisfies a defining property $\varphi$. Many of those formulas, however, will be nonstandard. Indeed, no pseudo-pointwise-definable infinite model can be actually pointwise definable, for then in $V^\B/U$ we would be able to define the standard cut by looking at the smallest formulas necessary to define the elements.

What information about $N$ can we extract in $V$ from the fact that $N$ is seen as pointwise definable inside $V^\B/U$? What is the nature of these pseudo-pointwise-definable models? I am unsure.

\printbibliography

\end{document}